\documentclass[final, a4paper, 10pt]{article}

\usepackage[scale=0.768]{geometry}
\usepackage[english]{babel}

\usepackage[mathscr]{euscript}
\usepackage{color}
\usepackage{fancyhdr}
\usepackage{bm}
\usepackage{hyperref}
\usepackage{cite}
\usepackage{showkeys}
\usepackage{subfigure}
\usepackage{float}

\usepackage{amsfonts}
\usepackage{amsmath}
\usepackage{amssymb}
\usepackage{amscd}
\usepackage{amsthm}
\usepackage[utf8]{inputenc}
\usepackage[T1]{fontenc}
\usepackage{accents}
\usepackage{graphicx}
\usepackage{array}

\usepackage{algorithm}
\usepackage{algorithmicx}
\usepackage{algpseudocode}
\usepackage{mathtools}
\DeclarePairedDelimiter\ceil{\lceil}{\rceil}

\newtheorem{theorem}{Theorem}
\newtheorem{definition}[theorem]{Definition}
\newtheorem{lemma}[theorem]{Lemma}
\newtheorem{corollary}[theorem]{Corollary}

\newtheorem{example}[theorem]{Example}
\newtheorem{remark}[theorem]{Remark}
\newtheorem{assumption}[theorem]{Assumption}

\newcommand*{\N}{\ensuremath{\mathbb{N}}}
\newcommand*{\Z}{\ensuremath{\mathbb{Z}}}
\newcommand*{\R}{\ensuremath{\mathbb{R}}}
\newcommand*{\C}{\ensuremath{\mathbb{C}}}
\renewcommand{\i}{\mathrm{i}}
\renewcommand{\phi}{\varphi}
\renewcommand{\rho}{{\varrho}}
\renewcommand{\epsilon}{{\varepsilon}}

\renewcommand{\d}[1]{\,\mathrm{d}#1 \,}

\newcommand{\D}{\mathcal{D}}

\newcommand{\T}{{\mathcal{T}}}

\newcommand{\V}{{\mathcal{V}}}

\newcommand{\K}{{\mathcal{K}}}

\newcommand{\p}{{per}}
\renewcommand{\L}{\mathcal{L}} 
\renewcommand{\Re}{\mathrm{Re}\,}
\renewcommand{\Im}{\mathrm{Im}\,}

\newcommand{\M}{{\mathcal{M}}}

\newcommand{\FF}{\mathbb{F}}
\newcommand{\UF}{\mathbb{UF}}

\newlength{\dhatheight}

\usepackage{color}
\newcommand{\high}[1]{{\color{red}{#1}}}
\definecolor{xl}{rgb}{0.8,0.2,0.3}

\usepackage{amsmath}
\usepackage{amsfonts}
\usepackage{amssymb}
\begin{document}
	
	\sloppy
	
\title{A spectral decomposition method  to approximate DtN maps in complicated waveguides}
\author{
Ruming Zhang\thanks{Institute of Applied and Numerical mathematics, Karlsruhe Institute of Technology, Karlsruhe, Germany
; \texttt{ruming.zhang@kit.edu}. }}
\date{}
\maketitle
	
\begin{abstract}In this paper, we propose a new spectral decomposition method to simulate  waves propagating in complicated waveguides. For the numerical solutions of waveguide scattering problems, an important task is to approximate the Dirichlet-to-Neumann map efficiently. From previous results, the physical solution can be decomposed into a family of generalized eigenfunctions, thus we can write the Dirichlet-to-Neumann map explicitly by these functions. From the exponential decay of the generalized eigenfunctions, we approximate the Dirichlet-to-Neumann (DtN) map by a finite truncation and the approximation is proved to converge exponentially. With the help of the truncated DtN map, the unbounded domain is truncated into a bounded one, and a variational formulation for the problem is set up in this bounded domain. The truncated problem is then solved by a finite element method. The error estimation is also provided for the numerical algorithm and numerical examples are shown to illustrate the efficiency of the algorithm.\\

 \noindent
 {\bf Keywords: spectral decomposition, complicated waveguides, Dirichlet-to-Neumann map}

\end{abstract}

\section{Introduction}

The numerical simulation of wave propagating in complicated unbounded waveguides is a challenging task, due to the existence of guided waves. To obtain the physical solution, the limiting absorption principle is always a standard way. 
In the past decades, a number of mathematicians as well as scientists from other disciplines have been working on this topic and several numerical methods have been proposed. In \cite{Joly2006}, the authors developed a numerical method to approximate the Dirichlet-to-Neumann maps in periodic waveguides, based on a numerical solution of an operator valued Riccati equation. The method is well-known and applied to other related topics in \cite{Fliss2009a,Fliss2013}. We would like to mention that in \cite{Fliss2021a}, the authors studied the same problem as in this paper with their method. Another important method, which is called the recursive doubling procedure, was proposed in \cite{Yuan07} first for exponential decaying solutions first. The method is later extended to further topics in \cite{Ehrhardt2009,Ehrhardt2009a} and the cases with guided waves are also included.

In recent years, mathematicians have made great improvements in the analysis of physical solutions in periodic waveguides. The structures of the are described by the radiation conditions, see \cite{Hoang2011,Fliss2015,Kirsc2017a,Hohag2013} for details. Although there are several versions of radiation conditions, they are equivalent in principle. Generally speaking, a physical solution is composed of a finite number of propagating modes and an evanescent part. With this property, a Bloch wave decomposition method was proposed in \cite{Dohna2018} for physical solutions in the joint of two different periodic half waveguides. In \cite{Zhang2019b,Zhang2021a}, the author also proposed the numerical method for (locally perturbed) periodic waveguides based on the Floquet-Bloch transform.

In this paper, we propose a spectral decomposition method to approximate the DtN map, which is given explicitly by the radiation condition, and then solve this problem numerically. This method is based on the spectral analysis of the translation operator for the physical solutions in periodic waveguides, in the author's previous paper \cite{Zhang2019a}. In this paper, the physical solution is decomposed into a finite number of propagating modes and a countable number of evanescent modes, which are generalized eigenfunctions for quasi-periodic problems. More precisely, compared to the existing results, the evanescent part of the physical solution can also be decomposed into a discrete set of evanescent modes. The evanescent modes decays exponentially, and the decay rate is determined by the corresponding eigenvalues. From the results in \cite{Hohag2013}, a distribution of the eigenvalues are studied and we can get a very accurate estimation of the number of generalized  eigenfunctions and their decay rate. These two papers  inspire the new method in this paper.

From the decomposition of the physical solutions, we  write out the Dirichlet-to-Neumann map explicitly by the generalized eigenfunctions. Then we approximate the Dirichlet-to-Neumann map by a finite truncation and the approximation also converges to the exact one exponentially. Then the domain is truncated into a bounded one from the Dirichlet-to-Neumann map, and a variational formulation for the problem is set up in this bounded domain. Finally we solve the problem by a finite element method. The error estimation is also provided for the numerical algorithm.

Note that this method was mentioned in Section 4.4, \cite{Ehrhardt2009}. However, the authors didn't work on this method due to the lack of results at that time. They also proposed two disadvantages according to the method. First, it is not easy to have an accurate approximation of small Floquet multipliers; second, the generalized eigenfunctions are not orthogonal and this makes it difficult to formulate the Dirichlet-to-Neumann map. Actually, since the generalized eigenfunctions decay very fast, we only need a small number of Floquet multipliers  and generalized eigenfunctions. Thus the difficulties make no problem in this case. We will discuss this in Section 3.3 in detail.

The rest of the paper is organized as follows. In the second section, the mathematical model for the problem and also important definitions and notations are introduced. In Section 3, some important results for periodic waveguides are recalled. With these information, the variational formulation is set up for the solution in Section 4. The numerical scheme and error estimation are organized in Section 5, and some numerical examples are shown in Section 6.

\section{Mathematical model and notations}

We consider waves propagating in a complicated waveguide $\Omega$. Suppose $\Omega$ is composed of three parts, a left half guide $\Omega_-:=(-\infty,-R)\times(-a,a)$, a right half guide $\Omega_+:=(R,+\infty)\times(-b,b)$  and $\Omega_0$ the part that joints $\Omega_-$ and $\Omega_+$.  Note that the structures of $\Omega_-$ and $\Omega_+$ are not necessarily the same. The domain $\Omega_0$ is assumed to be connected and its boundary $\partial\Omega_0$ is composed of finite number of Lipschitz continuous closed curves. For a visualization we refer to Figure \ref{fig:sample}.

\begin{figure}[H]
\begin{center}
 \includegraphics[width=0.9\textwidth]{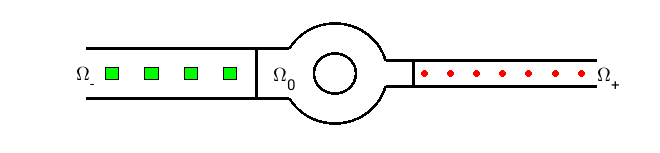}
 \caption{Structure of the complicated waveguide}
 \label{fig:sample}
 \end{center}
\end{figure}

The problem is modeled by the following Helmholtz equation:
\begin{equation}
 \label{eq:waveguide}
 \Delta u+k^2q u=f\text{ in } \Omega;\quad u=0\text{ on }\partial\Omega.
\end{equation}
The source term $f$ is assumed to be in $L^2(\Omega)$ and also compactly supported in $\Omega_0$. The refractive index $q\in L^\infty(\Omega)$ and is strictly positive:
\[
 q(x)\geq c>0\quad\text{ for all }\quad x\in\Omega.
\]
Moreover, the function $q$ is periodic in $\Omega_-$ and $\Omega_+$ in $x_1$-direction. For simplicity, let 
\[
 q(x)=\begin{cases}
       q_-(x),\quad\text{ when }x\in\Omega_-;\\
       q_+(x),\quad\text{ when } x\in\Omega_+;
      \end{cases}
\]
where $q_-$ is $L_-$-periodic and $q_+$ is $L_+$-periodic. Here the structures of $q_-$ and $q_+$ can be different. We denote the following periodicity cells:
\[
 \Omega^+_j:=(R+jL_+,R+(j+1)L_+)\times(-b,b)\quad\text{ and }\quad \Omega^-_j:=(-R+(j-1)L_-,-R+jL_-)\times(-a,a),
\]
where $ j=0,1,\dots,\infty$.
Define the line segments
\[
 \Gamma^+_j:=\{R+j L_+\}\times[-b,b]\quad\text{ and }\Gamma^-_j:=\{-R+(j-1) L_-\}\times[-a,a];\quad j=0,1,\dots.
\]
Then $\Gamma^+_{j}$ is the left boundary of $\Omega^+_j$ and $\Gamma^+_{j+1}$ is its right boundary; $\Gamma^-_{j}$ is the left boundary of $\Omega^-_j$ and $\Gamma^-_{j+1}$ is its right boundary. The half guides are composed by the cells:
\[
 \overline{\Omega_+}=\cup_{j=0}^\infty \overline{\Omega_j^+};\quad\overline{\Omega_-}=\cup_{j=-\infty}^0\overline{\Omega_j^-}.
\]
For simplicity, we can also extend $\Omega_+$ and $\Omega_-$ to the full guide
\[
 \widetilde{\Omega}_+:=\R\times(-b,b),\quad\widetilde{\Omega}_-:=\R\times(-a,a).
\]
Thus
\[
 \overline{\widetilde{\Omega}_+}=\cup_{j\in\Z} \overline{\Omega_j^+};\quad\overline{\widetilde{\Omega}_-}=\cup_{j\in\Z}\overline{\Omega_j^-}.
\]
At the same time, we can also extend $q_+$ and $q_-$ periodically into the full guides $\widetilde{\Omega}_+$ and $\widetilde{\Omega}_-$ and the extended functions are still denoted by $q_+$ and $q_-$.

We can also extend our problem to different boundary conditions and the Dirichlet boundary condition is just chosen as an example. Due to the existence of guided modes, the well-known {\em Limiting Absorption Principle (LAP)} is applied to the problem to get the unique physical solution. That is, we replace $k^2$ by $k^2+\i\epsilon$ for any $\epsilon>0$ in \eqref{eq:waveguide} and obtain the unique solution $u_\epsilon\in \widetilde{H}^1(\Omega)$ (here $\widetilde{H}^1(\Omega)$ is the subspace of $H^1(\Omega)$ with homogeneous Dirichlet data on $\partial\Omega$). Let $\epsilon\rightarrow 0^+$, then the limit of $u_\epsilon$ in $\widetilde{H}^1_{loc}(\Omega)$, which is called an LAP solution in this paper, is the solution we would like to simulate numerically. Especially, when $f\in L^2(\Omega_0)$ and $q\in L^\infty(\Omega)$, the function $u\in \widetilde{H}^2_{loc}(\Omega)$. 

To describe the solution obtained by the LAP, it is essential to introduce the radiation conditions in the half-guides $\Omega_-$ and $\Omega_+$. The DtN maps on the left (right) boundary of $\Omega_+$ ($\Omega_-$), which are defined by the radiation conditions,   play important roles. Thus we need to introduce some important definitions and results related to the DtN maps. 

\section{Periodic waveguide problem}

In this section, we  focus on the problem defined in a reference periodic waveguide $W:=\R\times(0,1)$ with the boundary $\partial W$. A number of radiation conditions have been introduced (see \cite{Hoang2011,Fliss2015,Kirsc2017a}) to characterize the LAP solutions and they are equivalent in principle. We first introduce the cell problems, and then conclude the radiation condition briefly. In the second part, we recall important spectral decomposition of the problem introduced in \cite{Zhang2019a}. At the end, we will estimate the distribution of the eigenvalues of the periodic waveguide.  We begin with the following problem:
\begin{equation}\label{eq:waveguideW}
 \Delta u+k^2 q u=g\text{ in }W;\quad u=0\text{ on }\partial W.
\end{equation}
The refractive index $q$ is strictly positive and periodic in $x_1$-direction:
\[
q(x)\geq c>0,\quad q(x_1+1,x_2)=q(x_1,x_2),\quad\forall\,x=(x_1,x_2)\in W.
\]
For simplicity, we also introduce the following notations. The periodicity cells and edges are denoted by:
\[
 W_j:=(j,j+1)\times(0,1);\quad \Gamma_j:=\{j\}\times(0,1).
\]

\subsection{Cell problems and the radiation condition}

We focus on a reference periodicity cell $W_0$ and  let  $\Sigma_-:=(0,1)\times\{0\}$ and $\Sigma_+:=(0,1)\times\{1\}$  be its lower and upper edges. We define space of $z$-quasi-periodic functions by:
\[
\V_z:= \widetilde{H}^1_z(W_0)=\left\{\phi\in \widetilde{H}^1(W_0):\,\phi\big|_{\Gamma_1}=z\phi\big|_{\Gamma_0}\right\}
\]
for any $z\in\C$. For any $z\in \C$, consider  the following problem in $\V_z$:
\begin{equation}
\label{eq:eig}
  \Delta w(z,x)+k^2 q w(z,x)=0\text{ in }W_0.
\end{equation}

Sometimes it is more convenient to use $\alpha:=-\i\log(z)$ instead of $z$. Here the logarithm function takes value in the branch cut $\left(-{\pi},{\pi}\right]+\i\R$. Then $|z|>1$ corresponds to $\Im(\alpha)<0$, and $|z|<1$ corresponds to $\Im(\alpha)>0$, and $|z|=1$ corresponds to $\alpha\in\left(-{\pi},{\pi}\right].$ 
Let  $V_\alpha:=\V_z$ when $z=e^{\i\alpha}$. For $\alpha=0$, the space $V_0$ is composed of periodic functions thus we let $V_\p:=V_0=\V_1$. Since the domains $W_n$ are with the same shape for all $n\in\Z$, we identify $\widetilde{H}^1_z(W_n)$ with $V_\alpha=\V_z$. With the same reason, we also identify $H^s(\Gamma_n)$ with $H^s(\Gamma_0)$ for $s=\pm 1/2$.  {Moreover, we also let $V:=\widetilde{H}^1(W_0)$.}

\begin{remark}
From now on, we use the notations $w(z,x)$ and $w(\alpha,x)$ to indicate the same function and this implies that $z=e^{\i\alpha}$.  Similarly for  $v_z$ and $v_\alpha$ in the following contents.
\end{remark}

For simplicity, we introduce a periodization technique. For a function $w(\alpha,\cdot)\in V_\alpha$, define:
\[
 v_\alpha(x):=e^{-\i\alpha x_1} w(\alpha,x)\in \,V_\p.
\]
From direct calculation, $v_\alpha$ satisfies
\begin{equation}
 \label{eq:per}
 \Delta v_\alpha+2\i\alpha\frac{\partial v_\alpha}{\partial x_1}+(k^2q-\alpha^2) v_\alpha=e^{-\i\alpha x_1}g(x)\text{ in }W_0;\quad \left.v_\alpha\right|_{\Gamma_1}= \left.v_\alpha\right|_{\Gamma_0}.
\end{equation}
The variational form of the above equation is to find $v_\alpha\in V_\p$ such that
\begin{equation}
 \label{eq:per_var}
 \int_{W_0}\left[\nabla v_\alpha\cdot\nabla\overline{\phi}-2\i\alpha\frac{\partial v_\alpha}{\partial x_1}\overline{\phi}-(k^2q-\alpha^2) v_\alpha\overline{\phi}\right]\d x=-\int_{W_0}e^{-\i\alpha x_1}g(x)\overline{\phi}(x)\d x
\end{equation}
holds for any $\phi\in V_\p$. This problem can be written as 
\begin{equation}
\label{eq:per_operator}
 \left<A_\alpha v_\alpha,\phi\right>_{V_\p}=\left<F_\alpha g,\phi\right>_{V_\p}
\end{equation}
where $A_\alpha:\,V_\p\rightarrow V_\p$ is a Fredholm operator depends analytically on $\alpha\in \C$ (see \cite{Kirsc2017a}) and $F_\alpha:\, L^2(W_0)\rightarrow V_\p$ depends analytically on $\alpha$. Moreover, $A_\alpha$ is self-adjoint when $\alpha$ is real.

From the analytic Fredholm theorem (Theorem VI.14, \cite{Reed1980}), all the points $z\in\C$ such that \eqref{eq:eig} has nontrivial solutions in $\V_z$ (or equivalently, those $\alpha=-\i\log(z)$ where $A_\alpha$ is not invertible in $V_\p$) compose a discrete set $\FF\subset C_\times:=\C\setminus\{0\}$, then $v$ depends analytically on $z\in\C_\times\setminus\FF$ and meromorphically on $z\in\C_\times$. For details we refer to Theorem 9 in \cite{Zhang2019a}. In the following, we introduce the set $\FF$ in details.

\begin{itemize}
 \item Define the set $$\UF:=\left\{z\in\FF:\,|z|=1\right\},$$ then it is  finite (can be empty).    Suppose $z\in\UF$, the nontrivial solution  $w(z,\cdot)\in \V_z$ is {\em a propagating mode or a  Bloch wave}. We recall some important facts introduced in \cite{Kirsc2017a}. Suppose for a fixed $z\in\UF$, $\mathcal{N}$ is the space spanned by all the nontrivial solutions of \eqref{eq:eig} in $\V_z$. Then $\mathcal{N}$ is a finite dimensional space with the dimension $m$. There is an orthonormal basis of $\mathcal{N}$, denoted by $\left\{\phi_j\in\mathcal{N}:\,j=1,2,\dots,m\right\}$ such that the following equations hold:
\begin{equation}
 \label{eq:ortho}
 -\i\int_{W_0}\frac{\partial\phi_j}{\partial x_1}\overline{\psi}\d x=k \lambda_j \int_{W_0}n\phi_j\overline{\psi}\d x\,\forall\,\psi\in\mathcal{N};\quad k\int_{W_0}n\phi_j\overline{\phi_{j'}}\d x=\delta_{j,j'}\,\forall\,j,j'=1,\dots,m,
\end{equation}
where $\lambda_j\in\R$ and $\delta_{j,j'}$ is the Kronecker delta function. In particular, when $\psi=\phi_j$,
\begin{equation}
\label{eq:lambdaj}
 -\i\int_{W_0}\frac{\partial\phi_j}{\partial x_1}\overline{\phi_j}\d x=\lambda_j.
\end{equation}
We can also decide the direction that $\phi_j$ propagates from the sign of the  parameter $\lambda_j$. When $\lambda_j>0$, $\phi_j$ propagates to the right; when $\lambda_j<0$, $\phi_j$ propagates to the left; while when $\lambda_j=0$, $\phi_j$ is a standing wave which has to be avoided (see Assumption \ref{asp1}).

\item Suppose $z\in\FF$ and $|z|<1$, then the nontrivial solution $w(z,\cdot)\in \V_z$  can be extended $z$-quasi-periodically to a solution in the full guide $W$ that decays exponentially when $x_1\rightarrow\infty$. Let $\mathcal{N}$ be the space spanned by all the generalized eigenfunctions of \eqref{eq:eig} in $V_z$ with the dimension $m$. Then it is spanned by the linear independent vectors $\phi_{1},\,\dots,\phi_{m}$, and the residue
\[
 \oint_{C_z}w(z,x)\d z=\sum_{\ell=1}^{m} c_\ell\,\phi_{\ell}(x)
\]
where $c_\ell\in\C$ are the coefficients and $C_z$ is the counterclockwise circle with center $z$ which encircles only one element $z\in\FF$.

\item When $z\in\FF$ with $|z|>1$, everything is similar but the corresponding generalized eigenfunctions are decaying exponentially when $x_1\rightarrow-\infty$.

\end{itemize}

We need the following assumption to guarantee that the LAP works.

\begin{assumption}
 \label{asp1}
 Assume that for $k$ and $q$ appear in this paper, there is no standing wave in either $\Omega_-$ or $\Omega_+$.
\end{assumption}


Finally we introduce some useful notations. Since $\UF$ is symmetric, let
\[
 \UF_+=\{z_1^+,z_2^+,\dots,z_{J}^+\};\quad  \UF_-=\{z_1^-,z_2^-,\dots,z_{J}^-\}
\]
where $z_j^+=\overline{z_j^-}$. Moreover, $z_j^+$ is associated with a rightward propagating wave and $z_j^-$ is associated with a leftward propagating wave. Note that it is possible that $\UF_+\cap\UF_-\neq\emptyset$. Then we define the sets 
\[
 \FF_+=\{z_1^+,z_2^+,\dots,z_{J}^+,z_{J+1}^+,\dots\};\quad  \FF_-=\{z_1^-,z_2^-,\dots,z_{J}^-,z_{J+1}^-,\dots\},
\]
where $\FF=\FF_+\cup\FF_-$. Moreover, $\left|z_j^+\right|<1$ and $\left|z_j^-\right|>1$ for all $j\geq J+1$, and the points are ordered by:
\[
\infty>\cdots\geq \left|z_{J+2}^-\right|\geq \left|z_{J+1}^-\right|> 1>\left|z_{J+1}^+\right|\geq \left|z_{J+2}^+\right|\geq \cdots>0.
\]
For each $z_j^\pm$ ($j\geq 1$), let $\mathcal{N}_j^\pm$ be the space spanned by all the nontrivial solutions of \eqref{eq:eig} in $\V_{z_j^\pm}$ with the dimension $m_j^\pm$. We also assume that the spaces are spanned by the following linear independent basis:
\[
 \mathcal{N}_j^+:={\rm span}\left\{\phi_{j,1}^+,\,\phi_{j,2}^+,\dots,\phi_{j,m_j^+}^+\right\};\quad \mathcal{N}_j^-:={\rm span}\left\{\phi_{j,1}^-,\,\phi_{j,2}^-,\dots,\phi_{j,m_j^-}^-\right\}.
\]

\subsection{ Spectral decomposition of the LAP solution}

In this subsection, we recall the spectral decomposition of the LAP solution developed in \cite{Zhang2019a}.  We recall the decomposition of the solution in the following theorem.

\begin{theorem}[Theorem 27, \cite{Zhang2019a}]
 \label{th:spec_decomp}
 Suppose Assumption \ref{asp1} holds. Then for $n\geq 1$,
 \begin{eqnarray}
  \label{eq:spec_decomp+}
  &&u(x_1+n,x_2)=
  \sum_{j=1}^{J}\sum_{\ell=1}^{m_j^+} \frac{\left<g,\phi_{j,\ell}^+\right>}{2\lambda^+_{j,\ell}}e^{\i n\alpha_j^+}\phi_{j,\ell}^+(x)+\sum_{j=J+1}^\infty {\rm Res}\left(w(z,x)z^{n-1},z=z_j^+\right);\\
  \label{eq:spec_decomp-} &&u(x_1-n,x_2)=
  \sum_{j=1}^{J}\sum_{\ell=1}^{m_j^-} \frac{\left<g,\phi_{j,\ell}^-\right>}{2\lambda^-_{j,\ell}}e^{-\i n\alpha_j^-}\phi_{j,\ell}^-(x)+\sum_{j=J+1}^\infty {\rm Res}\left(w(z,x)z^{-n-1},z=z_j^-\right),  
 \end{eqnarray}
where  $\lambda_{j,\ell}^\pm$ is obtained from \eqref{eq:lambdaj}, and the residue is defined as:
\[
 {\rm Res}\left(w(z,x)z^{m},z=z_0\right):=\oint_{|z-z_0|=\delta}w(z,x)z^{m}\d z,
\]
and $w(z,x)$ satisfies \eqref{eq:per} with $z=e^{\i\alpha}$ and $f$ be replaced by $g$, the disk with center $z_0$ and radius $\delta>0$ only contains one pole $z_0\in\FF$.

\end{theorem}

Choose two values $0<r<1<R$ such that there is no point in $\FF$ that lies on the circles with the center at $0$ and radius $r$ and $R$, i.e., $\{z\in\C:\,|z|=r\}\cap\FF=\{z\in\C:\,|z|=R\}\cap\FF=\emptyset$, we define the following integral for $n\geq 1$:
\[
 I_{n,r}^+\,g:=\oint_{|z|=r} w(z,x)z^{n-1} \d z;\quad  I_{-n,R}^-\,g:=\oint_{|z|=R} w(z,x)z^{-n-1} \d z.
\]
Thus $I_{n,r}^+$ and $I_{-n,R}^-$ are linear operators from $L^2(W_0)$ to $V$. 
The following result comes directly from the residue theorem in Banach spaces.

\begin{corollary}
 \label{cr:decomp_rl}
 Suppose Assumption \ref{asp1} holds. Then the following conditions hold for any $0<r<1<R$ and $n\geq 1$,
 \begin{eqnarray}
  \label{eq:decomp_rl+}
  && u(x_1+n,x_2)=
 \sum_{j=1}^{J_r^+}\sum_{\ell=1}^{m_j^+} c_{j,\ell}^+ e^{\i n\alpha_j^+}\phi_{j,\ell}^+(x)+I_{n,r}^+\,g;\\
  \label{eq:decomp_rl-}&& u(x_1-n,x_2)=
 \sum_{j=1}^{J_R^-}\sum_{\ell=1}^{m_j^-} c_{j,\ell}^- e^{-\i n\alpha_j^-}\phi_{j,\ell}^-(x)+I_{-n,R}^-\,g;
 \end{eqnarray}
where $J_r^+$ and $J_R^-$ are the positive integers such that
\[
 \left|z_{J_r^+}^+\right|> r>\left|z_{J_r^++1}^+\right|,\quad  \left|z_{J_R^-}^-\right|< R<\left|z_{J_R^-+1}^+\right|,
\]
and $c_{j,\ell}^\pm=\frac{\left<g,\phi_{j,\ell}^\pm\right>}{2\lambda_{j,\ell}}$ when $j=1,2,\dots,J$ and $\ell=1,2,\dots,m_j^\pm$.
\end{corollary}

Define the operators by:
\[
 \K_{n,r}^+\, g:=\sum_{j=1}^{J_r^+}\sum_{\ell=1}^{m_j^+}c_{j,\ell}^+ e^{\i n\alpha_j^+}\phi_{j,\ell}^+(x),\quad\text{ and }\quad \K_{-n,R}^-\, g:=\sum_{j=1}^{J_R^-}\sum_{\ell=1}^{m_j^-}c_{j,\ell}^- e^{-\i n\alpha_j^-}\phi_{j,\ell}^-(x).
\]
When $r=0$ and $R=\infty$, the two operators are defined in the following way
\[
  \K_{n,0}^+ g:=\sum_{j=1}^{\infty}\sum_{\ell=1}^{m_j^+}c_{j,\ell}^+ e^{\i n\alpha_j^+}\phi_{j,\ell}^+(x),\quad\text{ and }\quad \K_{-n,\infty}^- g:=\sum_{j=1}^{\infty}\sum_{\ell=1}^{m_j^-}c_{j,\ell}^- e^{-\i n\alpha_j^-}\phi_{j,\ell}^-(x).
\]
Then the following result comes immediately as a corollary of Theorem 18, \cite{Zhang2019a}.

\begin{lemma}
 Suppose $0<r<1<R<\infty$ are two numbers such that $\{z\in\C:\,|z|=r\text{ or }|z|=R\}\cap\FF=\emptyset$. Then
 \begin{equation}
 \label{eq:identity}
   \K_{n,0}^+=\K_{n,r}^++I_{n,r}^+ ;\quad \K_{-n,\infty}^- =\K_{-n,R}^-+I_{-n,R}^-\quad\text{ where }n\geq 1.
 \end{equation}
 For each fixed $n\geq 1$, there is a constant $C>0$ such that
 \[
  \left\|\K_{n,0}^+\, g\right\|_V\leq C\|g\|_{L^2(W_0)};\quad \left\|\K_{-n,\infty}^-\, g\right\|_V\leq C\|g\|_{L^2(W_0)}.
 \]

\end{lemma}

In the next theorem, we will estimate the terms $I_{n,r}^+g$ and $I_{n-,R}^-g$, when $r$ and $R$ take some special values.

\begin{theorem}
 \label{th:int_r}
 For each $m\in\N$, define $r_m :=\exp\left(-\pi\sqrt{\frac{m^2+(m+1)^2}{2}}\right)$ and $R_m:=\exp\left(\pi\sqrt{\frac{m^2+(m+1)^2}{2}}\right)$, then we have the following estimations for $n\geq 1$:
 \begin{equation}
  \label{eq:int_r}
  \left\|I_{n,r_m}^+\,g\right\|_{V}\leq C e^{-\pi(n-1)m}\|g\|_{L^2(W_0)}
 \end{equation}
 and
  \begin{equation}
  \label{eq:int_R}
  \left\|I_{-n,R_m}^- \,g\right\|_{V}\leq C e^{\pi(n-1)m}\|g\|_{L^2(W_0)}.
 \end{equation}
 Here the constant $C$ does not depend on $m,\,n$ and $g$.
\end{theorem}

\begin{proof}
We begin with the estimation of $I_{n,r_m}^+g$. 
From Lemma 26, \cite{Zhang2019a}, for any $\theta\in(-\pi,\pi]$, $z_m(\theta):=r_me^{\i\theta}\notin\FF$. Moreover, 
\begin{equation}
\label{eq:est_l2}
 \left\|w\left(z_m(\theta),\cdot\right)\right\|_{L^2(W_0)}\leq C e^{\pi m}m^{-1}\|g\|_{L^2(W_0)},
\end{equation}
where $C$ does not depend on $m$, $\theta$ and $g$. In \eqref{eq:per_var}, let $\alpha$ be replaced by $\alpha+\i\beta:=-{\i\log(z)}$ and   $v_\alpha=\phi=w\left(z_m(\theta),\cdot\right)$ in \eqref{eq:per_var}, take the real part, we get:
\begin{align*}
 \int_{W_0}\left[\left|\nabla w\left(z_m(\theta),x\right)\right|^2-2\i\alpha\frac{\partial w\left(z_m(\theta),x\right)}{\partial x_1}\overline{w\left(z_m(\theta),x\right)}-(k^2 q+\beta^2-\alpha^2)\left|w\left(z_m(\theta),x\right)\right|^2\right]\d x\\=-\Re\left(\int_{W_0}z^{-x_1}g(x)\overline{w\left(z_m(\theta),x\right)}\d x\right)
\end{align*}
With Young's inequality and H{\"o}lder's inequality,
\begin{align*}
\left\|\nabla w\left(z_m(\theta),\cdot\right)\right\|^2_{L^2(W_0)}&\leq \frac{1}{2}\left\|\frac{\partial w(z_m(\theta),\cdot)}{\partial x_1}\right\|^2_{L^2(W_0)}+2\alpha^2\left\|w(z_m(\theta),\cdot)\right\|^2_{L^2(W_0)}\\
&+\left\|k^2 q+\beta^2-\alpha^2\right\|_\infty\left\|w(z_m(\theta),\cdot)\right\|^2_{L^2(W_0)}+\left\|z^{- x_1}\right\|_\infty\|g\|_{L^2(W_0)}\left\|w(z_m(\theta),\cdot)\right\|_{L^2(W_0)}.
\end{align*}
Use the fact that $\left\|\frac{\partial w(z_m(\theta),\cdot)}{\partial x_1}\right\|_{L^2(W_0)}\leq \left\|\nabla w\left(z_m(\theta),\cdot\right)\right\|_{L^2(W_0)}$, since  $\alpha\in\left(-{\pi},{\pi}\right]$ and $\beta=O(m)$, 
\begin{align*}
 \left\|\nabla w\left(z_m(\theta),\cdot\right)\right\|^2_{L^2(W_0)}&\leq 2(k^2\|q\|_\infty+\beta^2+3\alpha^2)\left\|w(z_m(\theta),\cdot)\right\|^2_{L^2(W_0)}+r_m^{-1/2}\|g\|_{L^2(W_0)}\left\|w(z_m(\theta),\cdot)\right\|_{L^2(W_0)}\\
 &\leq  C(k^2\|q\|_\infty+\beta^2+3\alpha^2)e^{2\pi m}m^{-2}\|g\|_{L^2(W_0)}^2+C e^{2\pi m}m^{-1}\|g\|_{L^2(W_0)}^2\\
 &\leq C e^{\pi m}\|g\|_{L^2(W_0)}.
\end{align*}
Together with the estimation for the $L^2$-norm in \eqref{eq:est_l2}, we get the estimation:
\[
 \left\|w\left(z_m(\theta),\cdot\right)\right\|_{V}\leq C e^{\pi m}\|g\|_{L^2(W_0)}.
\]

Thus we plug this result to $I_{n,r_m}^+g$, and use  Minkowski's integral inequality (Theorem 202, \cite{Hardy1988}):
\begin{align*}
 \left\|I_{n,r_m}^+g\right\|_{V}&=\left\|\oint_{|z|=r_m}w(z,x)z^{n-1}\d z\right\|_{V}
 =\left\|\i r_m^n\int_{-\pi}^{\pi}w(z_m(\theta),x)e^{\i n\theta}\d \theta\right\|_{V}\\
 &\leq C r_m^n\int_{-\pi}^\pi \left\|w(z_m(\theta),\cdot)\right\|_{V_\p}\d\theta\leq C e^{-\pi(n-1)m}\|g\|_{L^2(W_0)}.
\end{align*}
Thus $I_{n,r_m}^+$ is a bounded linear operator from $L^2(W_0)$ to $V$. 

For $I_{-n,R_m}^-$, the proof is similar by replacing $z$ with $z^{-1}$ thus is omitted.
\end{proof}

\begin{lemma}
 \label{lm:compact}
 When $n\geq 2$, the operators $\K_{n,0}^+$ and $\K_{-n,\infty}^-$ are compact  from $L^2(W_0)$ to $V$.
\end{lemma}

\begin{proof}
First consider  $\K_{n,0}^+=\K_{n,r_m}^++I_{n,r_m}^+$ where $r_m:=\exp\left(-\pi \sqrt{\frac{m^2+(m+1)^2}{2}}\right)$. Note that $\K_{n,r_m}^+$ is a  finite rank operator, then we only need to study the properties of $I_{n,r_m}^+$.

From Theorem \ref{th:int_r}, for a fixed $n\geq 1$, there is a constant $C>0$ such that 
$$\|I^+_{n,r_m}\,g\|_{V}\leq C e^{-(n-1)\pi m}\|g\|_{L^2(\Omega_0^+)}\rightarrow 0,\quad\text{ as }m\rightarrow\infty.$$
 Thus $\K_{n,0}^+=\K^+_{n,r_m}+I^+_{n,r_m}$ is the limit of a sequence of finite rank operators $\K^+_{n,r_m}$, it is compact. The proof of $\K_{-n,\infty}^-$ is similar thus is omitted.
\end{proof}

\begin{figure}[H]
 \centering
 \includegraphics[width=0.45\textwidth]{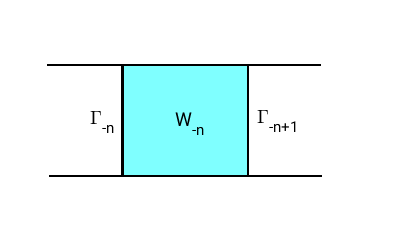}\includegraphics[width=0.45\textwidth]{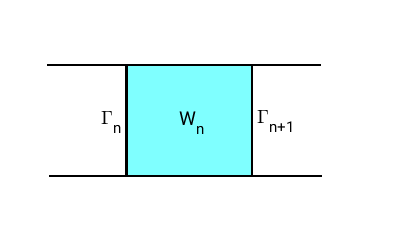}
 \label{fig:cellW}
 \caption{Periodicity cells in the periodic waveguide $W$.}
\end{figure}

For any $g\in L^2(W_0)$, $u_n=\K_{n,0}^+\,g$ and $u_{-n}=\K_{-n,\infty}^-\,g$ satisfy
\[
 \Delta u_n+k^2 q u_n =0\text{ in }W_n;\quad  \Delta u_{-n}+k^2 q u_{-n} =0\text{ in }W_{-n}
\]
with homogeneous Dirichlet boundary conditions on the upper and lower boundaries of $W_{\pm n}$. For the structures of $W_n$ and $W_{-n}$ we refer to Fig \ref{fig:cellW}. Since $u_n,\,u_{-n}\in V$, we conclude that
\[
 \mathcal{R}\left(\K_{n,0}^+\right),\,\mathcal{R}\left(\K_{-n,\infty}^-\right)\subset\left\{\phi\in V:\,\Delta \phi\in L^2(W_0)\right\}.
\]
From Theorem 5.5, \cite{Cakon2006}, $\left.\frac{\partial u_n}{\partial x_1}\right|_{\Gamma_n},\,\left.\frac{\partial u_{-n}}{\partial x_1}\right|_{\Gamma_{-n+1}}\in H^{-1/2}(\Gamma_0)$.


Define the trace operators 
\begin{align*}
 & \gamma^+_0\phi:=\left.\phi\right|_{\Gamma_n}\text{ and }\gamma^+_1\phi:=\left.\frac{\partial\phi}{\partial x_1}\right|_{\Gamma_n}\text{ for }\phi\in  \mathcal{R}\left(\K_{n,0}^+\right);\\& \gamma^-_0\phi:=\left.\phi\right|_{\Gamma_{-n+1}}\text{ and }\gamma^-_1\phi:=\left.\frac{\partial\phi}{\partial x_1}\right|_{\Gamma_{-n+1}}\text{ for }\psi\in \mathcal{R}\left(\K_{-n,\infty}^-\right)
\end{align*}
Then
\begin{align}\label{eq:der_positive_simplify}
& \left.\frac{\partial u}{\partial x_1}\right|_{\Gamma_{n}}=\gamma^+_1\left[ \K_{n,r}^+ +I_{n,r}^+ \right]g,\quad\text{ where }u\big|_{\Gamma_n}=\gamma_0^+\left[ \K_{n,r}^+ +I_{n,r}^+ \right]g;\\\label{eq:der_negative_simplify}
& \left.\frac{\partial u}{\partial x_1}\right|_{\Gamma_{-n+1}}=\gamma^-_1 \left[\K_{-n,R}^-+ I_{-n,R}^-\right] g,\quad\text{ where }u\big|_{\Gamma_{-n+1}}=\gamma_0^-\left[\K_{-n,R}^-+ I_{-n,R}^-\right] g.
\end{align}

Let $W_+:=(0,+\infty)\times(0,1)$ be a halfguide and $u_+$ be the LAP solution of 
\[
  \Delta u_++k^2q u_+=0\text{ in }W_+;\quad u_+=0\text{ on }\partial W_+\setminus\Gamma_0;\quad u_+=\phi\text{ on }\Gamma_0
\]
for a $\phi\in H^{1/2}(\Gamma_0)$, then define the DtN map $T^+$ by
\[
  T^+ \left[\phi\big|_{\Gamma_0}\right]=\left.\frac{\partial u_+}{\partial x_1}\right|_{\Gamma_0}.
\]
Similarly, let $W_-:=(-\infty,0)\times(0,1)$ and $u_-$ be the LAP solution that satisfies
\[
 \Delta u_-+k^2q u_-=0\text{ in }W_-;\quad u_-=0\text{ on }\partial W_-\setminus\Gamma_0;\quad u_-=\psi\text{ on }\Gamma_0.
\]
for a $\psi\in H^{1/2}(\Gamma_0)$, then define the DtN map $T^-$ by
\[
 T^- \left[\psi\big|_{\Gamma_{0}}\right]=-\left.\frac{\partial u_-}{\partial x_1}\right|_{\Gamma_{0}}.
\]

To make sure that $T^\pm$ are well defined, we need the following assumption.

\begin{assumption}[Assumption 5.1, \cite{Kirsc2017a}]
 \label{asp4}
 The only solution $u_\pm\in H^1(W_\pm)$ that satisfies
 \[
  \Delta u_\pm+k^2 q u_\pm=0\text{ in }W_\pm;\quad u_\pm=0\text{ on }\partial\,W_\pm
 \]
is the trivial one.
\end{assumption}

With above assumption, Theorem 5.4 in \cite{Kirsc2017a} showed that the half-guide problems are well-posed and the solutions $u_\pm\in H^1_{loc}(W_\pm)$ depend continuously on the boundary data $\phi,\psi\in H^{1/2}(\Gamma_0)$. The following lemma is a corollary of this theorem.

\begin{lemma}
 \label{lm:dtn_bdd}
 With Assumption \ref{asp1} and \ref{asp4}, the operators $T^\pm$ are bounded from $H^{1/2}(\Gamma_0)$ to $H^{-1/2}(\Gamma_0)$.
\end{lemma}

From the definitions of the operators $T^\pm$, we rewrite the boundary condition for LAP solutions in $W$. For an LAP solution $u$ of \eqref{eq:waveguideW},
\begin{equation}
\label{eq:bc}
 \left.\frac{\partial u}{\partial x_1}\right|_{\Gamma_n}=T^+\gamma_0^+\left[\K_{n,r}^++I_{n,r}^+\right]g;\quad \left.\frac{\partial u}{\partial x_1}\right|_{\Gamma_{-n+1}}=T^-\gamma_0^-\left[\K_{-n,R}^-+I_{-n,R}^-\right]g,
\end{equation}
where
\[
 \left. u\right|_{\Gamma_n}=\gamma_0^+\left[\K_{n,r}^++I_{n,r}^+\right]g;\quad \left. u\right|_{\Gamma_{-n+1}}=\gamma_0^-\left[\K_{-n,R}^-+I_{-n,R}^-\right]g.
\]


\subsection{Estimation of number of generalized eigenfunctions}

From \eqref{eq:bc}, \eqref{eq:identity} and the definitions of $\K_{n,0}^+$ and $\K_{-n,\infty}^-$, the DtN maps $T^\pm$ are defined by infinite number of generalized eigenfunctions. In this subsection, we would like to estimate the dimension of the eigenspace related to the eigenvalues lying in the rectangle 
\begin{equation}
\label{eq:rect}
 D_m^+:=[-\pi,\pi]+\i[0,
\pi m)\text{ or }D_m^-:=[-\pi,\pi]+\i(-\pi m,0],
\end{equation}
when $m$ is a positive integer.  The estimation is based on Proposition 4.2 in \cite{Hohag2013}.

Recall that the operator $A_\alpha$, which is defined by the left hand side in \eqref{eq:per_var},  it depends on the refractive index $q$. So we rewrite it as $A_q(\alpha)$. 
Then $A_0(\alpha)$ is related to the case that $q=0$, which comes directly from the Laplacian operator. All these $\alpha\in \C$ such that $A_0(\alpha)$ is not invertible are given explicitly by:
\[
 \alpha_{j,\ell}=2\pi j\pm\i\pi \ell;\quad \phi=\exp\left(\i\left[2\pi j\pm \i\pi \ell\right]x_1\right)\sin(\pi \ell x_2),\quad j\in\Z,\,\ell=1,2,\dots,\infty.
\] Since we are only interested in the values in $(-\pi,\pi]+\i\R$, let's focus on the values with $j=0$. Moreover, we only focus on the domain that lies above the real axis since the domain below is symmetric. So we define the reference points as
\[
 \alpha_n^r=\i\pi n,\quad n=1,2,\dots,\infty.
\] 
For simplicity, let the dimension of the eigenspace of $A_q(\alpha)$ at the point $\alpha_0$ be denoted by $\mathfrak{N}(A_q(\alpha);\alpha_0)$. Suppose $\Gamma$ is a closed curve that encircles the points $\alpha_1,\dots,\alpha_p$, then let $\mathfrak{N}(A_q(\alpha);\Gamma)=\sum_{j=1}^q\mathfrak{N}(A_q(\alpha);\alpha_q)$.

Before the introduction of the result in Proposition 4.2 in \cite{Hohag2013}, we define the open disc 
$$
\D_n:=\left\{z\in\C:\,\left|z-\alpha_n^r\right|<\frac{2k^2\|q\|_\infty}{\pi(2n-1)}\right\}
$$
and the closed rectangle
$$
\mathcal{N}_n:=[-\pi,\pi]+\i\left[\left(n-\frac{1}{2}\right)\pi,\left(n+\frac{1}{2}\right)\pi\right].
$$

\begin{theorem}[Proposition 4.2 in \cite{Hohag2013}]
\label{th:est_eigenvalue}
 Let $N_0$ be a positive integer which is defined by
 \[
  N_0\geq\min\left\{n\in\N:\,\frac{\pi(2n-1)}{2}\geq\frac{k^2\|q\|_\infty}{\pi}\right\},
 \]
 take $n\geq N_0$. Then
 \[
  \mathfrak{N}(A_q(\alpha);\partial\D_n)=\mathfrak{N}(A_q(\alpha);\mathcal{N}_n)=1.
 \]

\end{theorem}

From Theorem \ref{th:est_eigenvalue}, we  estimate $\mathfrak{N}(A_q(\alpha);\partial D_m^+)$. Let the number of points in $\FF$ lying in the domain $D_{N_0-1/2}^+$ be $M_0$, then we can estimate the number  for a general $m>0$. Then
\[
  \mathfrak{N}(A_q(\alpha);\partial D_m^+)\leq\begin{cases}
 M_0,\quad\text{ when }     m\leq N_0-1/2;\\
 M_0+\ceil*{m-N_0+\frac{1}{2}},\quad\text{ when }m> N_0-1/2;
                                              \end{cases}
\]
 where $\ceil*{s}$ is the largest integer which is no larger than $s$. So from above arguments, it is concluded that
 \[
  \mathfrak{N}(A_q(\alpha);\partial D_m^+)=O(m).
 \]
 Thus the dimension of the eigenspace related to eigenvalues lying in $D_m^+$ is $O(m)$. The result is the same for $D_m^-$.

\section{Variational formulation}

In this section, we introduce a variational formulation for the LAP solution $u$ to the equation \eqref{eq:waveguide} in the open bounded domain $D$ such that:
\[
 \overline{D}=\overline{\Omega_1^-\cup\Omega_0^-\cup\Omega_0\cup\Omega_0^+\cup\Omega_1^+}.
\]
Since we have to use the results  for periodic  waveguides, we first extend the LAP solution $u$ restricted in $\Omega_+$ ($\Omega_-$) to the full waveguide $\widetilde{\Omega}_+$ ($\widetilde{\Omega}_-$).  Let $\mathcal{X}_+$ and $\mathcal{X}_-$ be two smooth functions which satisfy
\[
 \mathcal{X}_+(x_1)=\begin{cases}
                   1,\quad x_1\geq R+L_+;\\
                   0,\quad x_1\leq R;\\
                   \text{smooth},\quad\text{otherwise};
                  \end{cases}\quad \mathcal{X}_-(x_1)=\begin{cases}
                   1,\quad x_1\leq -R-L_-;\\
                   0,\quad x_1\geq -R;\\
                   \text{smooth},\quad\text{otherwise}.
                  \end{cases}
\]
Let $\widetilde{u}_+(x):=\mathcal{X}_+(x_1)u(x)$ and extend it by $0$ into $\widetilde{\Omega}_+$. Then it is a solution of
\begin{equation}
 \label{eq:half_omega+}
 \Delta\widetilde{u}_++k^2 q_+ \widetilde{u}_+=\M^+ u\left(:=2\mathcal{X}'_+(x_1)\frac{\partial u(x)}{\partial x_1}+\mathcal{X}''_+(x_1)u(x)\right)\text{ in }\widetilde{\Omega}_+;\quad \widetilde{u}_+=0\text{ on }\partial \widetilde{\Omega}_+.
\end{equation}
Similarly, let $\widetilde{u}_-(x):=\mathcal{X}_-(x_1)u(x)$ and extend it by $0$ into $\widetilde{\Omega}_-$, then it satisfies
\begin{equation}
 \label{eq:half_omega-}
 \Delta\widetilde{u}_-+k^2 q_- \widetilde{u}_-=\M^- u\left(:=2\mathcal{X}'_-(x_1)\frac{\partial u(x)}{\partial x_1}+\mathcal{X}''_-(x_1)u(x)\right)\text{ in }\widetilde{\Omega}_-;\quad \widetilde{u}_-=0\text{ on }\partial \widetilde{\Omega}_-;
\end{equation}
The operator $\M^+$ ($\M^-$) is bounded from $\widetilde{H}^1(\Omega_0^+)$ ($\widetilde{H}^1(\Omega_0^-)$) to $L^2(\Omega_0^+)$ ($L^2(\Omega_0^-)$), and there is a constant $C>0$ such that
\[\left\|\M^+ u\right\|_{L^2(\Omega_0^+)}\leq C\|u\|_{\widetilde{H}^1(\Omega_0^+)},\quad
\left\|\M^- u\right\|_{L^2(\Omega_0^-)}\leq C\|u\|_{\widetilde{H}^1(\Omega_0^-)}.
\]

From Corollary \ref{cr:decomp_rl}, the solutions $\widetilde{u}_+$ and $\widetilde{u}_-$ are the unique LAP solutions of \eqref{eq:half_omega+} and \eqref{eq:half_omega-} in $\widetilde{\Omega}_+$ and $\widetilde{\Omega}_-$ with source terms $\M^+ u$ and $\M^-u$, respectively. Moreover, from their definitions,
\[
\left. \widetilde{u}_+\right|_{\Omega_+\setminus\overline{\Omega_0^+}}=\left.u\right|_{\Omega_+\setminus\overline{\Omega_0^+}};\quad \left. \widetilde{u}_-\right|_{\Omega_-\setminus\overline{\Omega_0^-}}=\left.u\right|_{\Omega_-\setminus\overline{\Omega_0^-}}.
\]
From \eqref{eq:bc}, 
\begin{align*}
 &\left.\frac{\partial \widetilde{u}_+}{\partial x_1}\right|_{\Gamma_2^+}=T^+\gamma^+_0 \left[\K_{r}^+ + I^+_r \right]\M^+ u,\quad\text{ where } \left.\widetilde{u}_+\right|_{\Gamma_2^+} =\gamma_0^+\left[\K_{r}^+ + I^+_r \right]\M^+ u;\\
 & \left.\frac{\partial \widetilde{u}_-}{\partial x_1}\right|_{\Gamma_{-1}^-}=T^-\gamma^-_0 \left[\K_{R}^- + I_{R}^- \right]\M^- u,\quad\text{ where } \left.\widetilde{u}_-\right|_{\Gamma_{-1}^-}=\gamma^-_0 \left[\K_{R}^- + I_{R}^- \right]\M^- u;
\end{align*}
where $\K_r^+=\K_{2,r}^+$, $I^+_{2,r}$, $\K_R^-=\K_{-2,R}^-$ and $I_R^-=I^-_{-2,R}$.

With these boundary conditions, the problem can be formulated in the bounded domain $D$. Let
\[
 V_D:=\left\{\phi\in H^1(D):\,\phi=0\text{ on }\partial D\cap \partial\Omega\right\}.
\]
Then we look for $u\in V_D$ such that
\begin{equation}
 \label{eq:var_D}
 \int_D\left[\nabla u\cdot\nabla\overline{\phi}-k^2  n u\overline{\phi}\right]\d x-\int_{\Gamma_2^+}\frac{\partial u}{\partial x_1}\overline{\phi}\d s+\int_{\Gamma_{-1}^-}\frac{\partial u}{\partial x_1}\overline{\phi}\d s=-\int_D f\overline{\phi}\d x
\end{equation}
holds for all $\phi\in V_D$. From Riesz representation theorem, we can define the following operators in $V_D$ by:
\begin{align*}
  \left<A u,\phi\right>_{V_D}=\int_D\left[\nabla u\cdot\nabla\overline{\phi}-k^2  n u\overline{\phi}\right]\d x;\quad
 \left<D^+ u,\phi\right>_{V_D}=\int_{\Gamma_2^+}\frac{\partial u}{\partial x_1}\overline{\phi}\d s;\quad
 \left<D^- u,\phi\right>_{V_D}=-\int_{\Gamma_{-1}^-}\frac{\partial u}{\partial x_1}\overline{\phi}\d s.
\end{align*}
Then \eqref{eq:var_D} is written as the equivalent form:
\begin{equation}
 \label{eq:var_exa}
 (A-D^+-D^-)u=F,
\end{equation}
where $F\in V_D$ such that $\left<F,\phi\right>_{V_D}=-\int_D f\overline{\phi}\d x$.
Since the operator $A$ is Fredholm, we still need to study the  operators $D^\pm$.

Let $\gamma^+:\,\phi\mapsto\phi\big|_{\Gamma_2^+}$ and $\gamma^-:\,\phi\mapsto\phi\big|_{\Gamma_{-1}^-}$ be the trace operators from $V_D$ to $H^{1/2}(\Gamma_2)$ and $H^{1/2}(\Gamma_{-2})$. Then from a duality argument,
\[
 D^+=\left(\gamma^+\right)^*T^+\gamma_0^+\left[\K_r^++I_r^+\right]\M^+\text{ and } D^-=\left(\gamma^-\right)^*T^-\gamma_0^-\left[\K_R^-+I_R^-\right]\M^-.
\]
From Lemma \ref{lm:dtn_bdd}, $T^\pm$ are bounded. From the boundedness of $\gamma^\pm$, $\gamma_0^\pm$ and $\M^\pm$, with the results in Lemma \ref{lm:compact}, $D^\pm$ are compact operators.  Then $A-D^+-D^-$ is a Fredholm operator. Thus it is invertible if and only if it is uniquely solvable. To guarantee that the original problem \eqref{eq:waveguide} has a unique solution, we make the following assumption.

\begin{assumption}
 \label{asp2}
 For given $k>0$ and $n\in L^\infty(\Omega_0)$, there is no eigenfunctions of \eqref{eq:waveguide}.
\end{assumption}

With Assumption \ref{asp2}, the following well-posedness result is obvious from the Fredholm alternative.

\begin{theorem}
 \label{th:well_posedness}Suppose Assumptions \ref{asp1}, \ref{asp2} and \ref{asp4} hold.
 The problem \eqref{eq:var_exa} has a unique solution which satisfies the radiation condition defined in Theorem \ref{th:spec_decomp}.
\end{theorem}

\begin{theorem}
 \label{th:equivalent}With Assumptions \ref{asp1}, \ref{asp2} and \ref{asp4}, the unique solution of \eqref{eq:var_exa} in $V_D$ is equivalent to the LAP solution of \eqref{eq:waveguide}.
\end{theorem}

\begin{proof}
 
 Suppose $u$ is the LAP solution of \eqref{eq:waveguide}. From the constructions above,  the solution $u\big|_{D}$ is an element in $V_D$ and satisfies \eqref{eq:var_D} which is equivalent to \eqref{eq:var_exa}. 
 
 Suppose $u$ is the solution of \eqref{eq:var_exa}, then $u$ satisfies $\Delta u+k^2 qu=0$ in $\Omega_0^+\cup\Omega_1^+$ with $u=0$ on the lower and upper boundaries of $\Omega_+$.  Let $\widetilde{u}_+(x):=\mathcal{X}_+(x_1)u(x)$ and extend it by $0$ to $\widetilde{\Omega}_+\setminus\Omega_+$. Define the domain $\widetilde{\Omega}_+^-$ by $\widetilde{\Omega}_+\setminus\left(\Omega_+\setminus(\overline{\Omega_0^+\cup\Omega_1^+})\right)$, then $\widetilde{u}_+$ satisfies
 \[
  \Delta \widetilde{u}_++k^2 q_+ \widetilde{u}_+=\M^+ u\text{ in }\widetilde{\Omega}_+^-;\quad \widetilde{u}_+=0\text{ on }\partial\widetilde{\Omega}_+^-\cap\partial\Omega.
 \]
  From definition of $D^+$ and the fact that $u=\widetilde{u}_+$ in $\Omega_1^+$, 
 \[
 \left. \frac{\partial \widetilde{u}_+}{\partial x_2}\right|_{\Gamma_2^+}=T^+\gamma^+_0\left[\K_{r}^++I_{r}^+\right]\M^+ u,\quad\text{ where }\left. \widetilde{ u}_+\right|_{\Gamma_2^+}=\gamma^+_0\left[\K_{r}^++I_{r}^+\right]\M^+ u.
 \]
 We define the function $\widetilde{u}_+$ by
 \[
  \widetilde{u}_+(x_1+n,x_2):=\left[\K_{n,r}^++I_{n,r}^+\right]\M^+ u,\quad n=2,3,\dots.
 \]
Then $\widetilde{u}_+$ as well as its derivative are continuous across the boundary $\Gamma_2^+$. This implies that $\widetilde{u}_+$ is extended to be an LAP solution in $H^1_{loc}\left(\widetilde{\Omega}_+\right)$. Since $u=\widetilde{u}_+$ in $\Omega_1^+$, $u$ is also extended to $\Omega_+$ and satisfies the radiation conditions \eqref{eq:decomp_rl+}-\eqref{eq:decomp_rl-}. 

With a similar process, we also extend $u$ to $\Omega_-$ as an LAP solution. Thus $u$ is the LAP solution of \eqref{eq:waveguide}. From Assumption \ref{asp2}, there is no eigenfunction of \eqref{eq:waveguide}. Thus the LAP solution for the problem is unique. So $u$ is the unique LAP solution of \eqref{eq:waveguide}.

\end{proof}

From perturbation theory, we approximate the problem \eqref{eq:var_D} by the following modified problem:
\begin{equation}
 \label{eq:var_D_approx}
 \int_D\left[\nabla u\cdot\nabla\overline{\phi}-k^2  n u\overline{\phi}\right]\d x-\left<T^+\gamma_0^+\K_r^+\M^+u,{\phi}\right>_{L^2(\Gamma_2^+)}+\left<T^-\gamma_0^-\K_R^-\M^-u,{\phi}\right>_{L^2(\Gamma_{-1}^-)}=-\int_D f\overline{\phi}\d x
\end{equation}

\begin{theorem}
 \label{th:uni_approx}
 Let $r_m:=\exp\left(-\pi L_+\sqrt{\frac{m^2+(m+1)^2}{2}}\right)$ and $R_m:=\exp\left(\pi L_-\sqrt{\frac{m^2+(m+1)^2}{2}}\right)$ . Then the variational problem \eqref{eq:var_D_approx} has a unique solution in $V_D$ and the solution $u_m$ satisfies
 \begin{equation}\label{eq:esti_approx}
  \left\|u_m-u\right\|_{V_D}\leq C e^{-c m}\|u\|_{V_D},
 \end{equation}
where $u$ is the solution of the original problem \eqref{eq:var_D}.
\end{theorem}
 The proof comes directly from the convergence rate of $I^+_{r_m}$ and $I^-_{R_m}$ in Theorem \ref{th:int_r}, thus is omitted.
 
 \vspace{0.5cm}
 
 We rewrite the boundary terms in \eqref{eq:var_D_approx} to be fit for numerical simulations. From \eqref{eq:bc} and the definitions of $\K_r^+$ and $\K_R^-$, 
 \begin{align*}
  &T^+ u_m=\sum_{j=1}^{J_{r_m}^+}\sum_{\ell=1}^{m_j^+}c^+_{j,\ell} \frac{\partial\phi_{j,\ell}^+(x)}{\partial x_1},\text{ where }u_m(x)=\sum_{j=1}^{\infty}\sum_{\ell=1}^{m_j^+}c^+_{j,\ell}\phi_{j,\ell}^+(x)\quad\text{ on }\Gamma_2^+;\\
 & T^- u_m=\sum_{j=1}^{J_{R_m}^-}\sum_{\ell=1}^{m_j^-}c^-_{j,\ell} \frac{\partial\phi_{j,\ell}^-(x)}{\partial x_1},\text{ where }u_m(x)=\sum_{j=1}^{\infty}\sum_{\ell=1}^{m_j^-}c^-_{j,\ell}\phi_{j,\ell}^-(x)\quad\text{ on }\Gamma_{-1}^-.
 \end{align*}
Thus finally \eqref{eq:var_D_approx} becomes
\begin{equation}
 \label{eq:var_D_comput}
 \begin{aligned}
 \int_D\left[\nabla u\cdot\nabla\overline{\phi}-k^2  n u\overline{\phi}\right]\d x&-\sum_{j=1}^{J_r^+}\sum_{\ell=1}^{m_j^+}c^+_{j,\ell} \int_{\Gamma_2^+}\frac{\partial\phi_{j,\ell}^+(x)}{\partial x_1}\overline{\phi}(x)\d s\\&+\sum_{j=1}^{J_r^-}\sum_{\ell=1}^{m_j^-}c^-_{j,\ell} \int_{\Gamma_{-1}^-}\frac{\partial\phi_{j,\ell}^-(x)}{\partial x_1}\overline{\phi}(x)\d s=-\int_D f\overline{\phi}\d x
 \end{aligned}
\end{equation}

From Theorem \ref{th:est_eigenvalue}, to obtain an approximation $u_m$ with the error about $e^{-cm}$, actually we only need $O(m)$ number of generalized eigenfunctions to construct the variational problem \eqref{eq:var_D_comput}. Thus the computational complexity in the formulation of the DtN map is expected to be small.

\section{Numerical implementation}

In this section, we introduce the numerical method to solve the problem \eqref{eq:var_D_comput}. The process consists of two steps. In the first section, we approximate the eigenvalues and eigenfunctions $\lambda_{j,\ell}^\pm$ and $\phi^+_{j,\ell}$ by the spectral method; in the second section, we approximate the problem \eqref{eq:var_D_approx} by a high order finite element method.
	
\subsection{Numerical approximation of eigenvalues and eigenfunctions}

In this subsection, we introduce the numerical approximation of the eigenvalues and eigenfunctions $\lambda_{j,\ell}^\pm$ and $\phi^\pm_{j,\ell}$ by the spectral method. For simplicity, we still take the reference waveguide $W$ as an example.

Recall $W_0:=(0,1)\times(0,1)$,  $\Gamma_0:=\{0\}\times(0,1)$, $\Gamma_1:=\{1\}\times(0,1)$, $\Sigma_-:=(0,1)\times\{0\}$ and $\Sigma_+:=(0,1)\times\{1\}$. 
We are looking for nontrivial solutions in $V_\p$ such that it satisfies 
\begin{equation}
\label{eq:eig_per}
 \Delta v+2\i\alpha\frac{\partial v}{\partial x_1}+(k^2q-\alpha^2)v=0\text{ in }W_0;\quad v=0\text{ on }\Sigma_-\cup\Sigma_+.
\end{equation}
We expand $v$  by the Fourier series
\begin{equation}\label{eq:v_four}
 v=\sum_{j=-\infty}^\infty\sum_{\ell=1}^\infty v_{j,\ell} \exp(\i 2\pi j x_1)\sin(\pi\ell x_2),\text{ where }v_{j,\ell}=2\int_{\Omega_0}v(x)\exp(-\i 2\pi j x_1)\sin(\pi\ell x_2)\d x.
\end{equation}
From direct calculation of the Fourier transform,
\begin{equation}
\label{eq:eig_diff}
 \Delta v+2\i\alpha\frac{\partial v}{\partial x_1}-\alpha^2v=\sum_{j=-\infty}^\infty\sum_{\ell=1}^\infty(-4\pi^2 j^2-\pi^2\ell^2-4\pi j\alpha-\alpha^2) v_{j,\ell} \exp(\i 2\pi j x_1)\sin(\pi\ell x_2).
\end{equation}
The only difficulty lies in the term $qv$. We first extend  $q$ to be an even function with respect to $x_2\in[-1,1]$, and then extend it periodically to $x_2\in\R$. Thus it has the Fourier series
\[
q=\sum_{j=-\infty}^\infty\sum_{\ell=0}^\infty q_{j,\ell} \exp(\i 2\pi j x_1)\cos(\pi\ell x_2),\text{ where } q_{j,\ell}=c_\ell\int_{\Omega_0}q(x)\exp(-\i 2\pi j x_1)\cos(\pi\ell x_2)\d x,
\]
here $c_\ell=2$ when $\ell\neq 0$ and $c_0=1$. 

Since $qv=0$ when $x_2=0,1$, it can be extended to be an odd $2$-periodic function in $x_2$ direction. Thus it is spanned by 
\begin{equation}
\label{eq:eig_conv}
 qv=\sum_{j=-\infty}^\infty\sum_{\ell=1}^\infty w_{j,\ell} \exp(\i 2\pi j x_1)\sin(\pi\ell x_2),
\end{equation}
where
\[
w_{j,\ell}= 2 \int_{\Omega_0} q(x)v(x)e^{-\i 2\pi j x_1}\sin(\pi\ell x_2)\d x=\frac{1}{2}\sum_{j'=-\infty}^\infty\sum_{\ell'=1}^\infty v_{j',\ell'}\left[q_{j-j',\ell-\ell'}-q_{j-j',\ell+\ell'}\right].
\]
Put \eqref{eq:eig_diff} and \eqref{eq:eig_conv} into \eqref{eq:eig_per}, then the coefficients satisfy the following equations:
\begin{equation}
\label{eq:system}
 \left(-4\pi^2 j^2-\pi^2\ell^2-4\pi\alpha j-\alpha^2\right)v_{j,\ell}+ \frac{k^2}{2}\sum_{j'=-\infty}^\infty\sum_{\ell'=1}^\infty v_{j',\ell'}\left[q_{j-j',\ell-\ell'}-q_{j-j',\ell+\ell'}\right]=0.
\end{equation}

To solve the above quadratic eigenvalue problem, we truncate the series \eqref{eq:v_four} for a large $N$:
\begin{equation}\label{eq:v_four_N}
 v_N=\sum_{j=-N}^N\sum_{\ell=1}^N v_{j,\ell}^N \exp(\i 2\pi j x_1)\sin(\pi\ell x_2).
\end{equation}
Let $V_N:=\left(v_{j,\ell}^N\right)_{j=-N,\cdots,N,\,\ell=1,\dots,N}$ be the vector of coefficients, then \eqref{eq:system} now becomes 
\[
 (B+\alpha A-\alpha^2 I)V_N=0,
\]
where $I$ is the identity matrix, and $B$, $A$ are matrices which comes directly from \eqref{eq:system}. To solve the quadratic eigenvalue problem, we formulate the following linearized problem. Let $W_N:=\alpha V_N$, then
\begin{equation}
\label{eq:eig_pb}
\L\left(\begin{matrix}
        V_N \\ W_N
       \end{matrix}
\right):= \left(\begin{matrix}
        B & A\\
        I & 0
       \end{matrix}
\right) \left(\begin{matrix}
        V_N \\ W_N
       \end{matrix}
\right)= \alpha \left(\begin{matrix}
        V_N \\ W_N
       \end{matrix}
\right)
\end{equation}
By solving the above linear eigenvalue problem, we can find out the corresponding eigenvalues and eigenfunctions.  

To approximate the DtN map, we need to find out all the eigenvalues and eigenfunctions in $D_M^+$ and $D_M^-$ defined in \eqref{eq:rect}, for some $M>0$.  For the real eigenvalues, we need to apply \eqref{eq:ortho}-\eqref{eq:lambdaj} to find out all the rightward and leftward propagating modes. Finally we obtain the modes $\phi_{j,\ell}^{M,N,+}$, where $\ell=1,2,\dots,M_j^+$ and $j=1,2,\dots,J_r^+$ that associate with $\Omega_+$, and modes $\phi_{j,\ell}^{M,N,-}$, where $\ell=1,2,\dots,M_j^-$ and $j=1,2,\dots,J_R^-$ that associate with $\Omega_-$. For simplicity, we reorder the modes and denote them by $$\phi_m^{M,N,+}\text{ where }m=1,2,\dots,M^+;\quad \phi_m^{M,N,-}\text{ where }m=1,2,\dots,M^-.$$ 
Now we are prepared to discretize \eqref{eq:var_D_comput} by a finite element method.



\subsection{Finite element method}

To discretize \eqref{eq:var_D_comput} by the finite element method, we generate quasi-regular triangular meshes $\M_h$ in $\Omega$, and the meshsize is $h>0$ and sufficiently small. Then we use the cubic Lagrangian element to compute the solution of \eqref{eq:var_D_comput}. 
 Let $\left\{\phi_j^{(h)}(x),j=1,2,\dots,N_h\right\}$ be the set of piecewise  quadratic basis functions with homogeneous Dirichlet boundary conditions on $\partial \Omega$, based on the mesh $\M_h$.  We approximate $u$ by 
\[
 u_{h,M,N}(x)=\sum_{j=1}^{N_h}\widehat{u}(j)\phi_j^{(h)}(x),
\]
where $\widehat{u}(j)\in\C$ are the coefficients. To discretize \eqref{eq:var_D_comput}, we still need to find out the coefficients $c_{m}^{M,N,\pm}$ from the above approximation $u_{h,M,N}\big|_{\Gamma_2^+}$ and $u_{h,M,N}\big|_{\Gamma_{-1}^-}$. 
Take $u_{h,M,N}\big|_{\Gamma_2^+}$ for example, it is sufficient to decompose $$\left.\phi_n^{(h)}\right|_{\Gamma_2^+}=\sum_{m=1}^{M^+}c^{M,N,+}_{m,n} \left.\phi_{m}^{M,N,+}(x)\right|_{\Gamma_2^+}$$
for those $\phi_n^{(h)}$ such that $\left.\phi_n^{(h)}\right|_{\Gamma_2^+}\neq 0$ almost everywhere, 
where $\phi_m^{M,N,+}$ are the eigenfunctions obtained in the previous subsection. 
 Thus $c_m^{M,N,+}$ can be obtained by solving the following linear system:
\begin{equation}
\label{eq:linear_sys}
 \left(\begin{matrix}
  a_{11}   & a_{12}&\cdots & a_{1,M^+}\\
  a_{21} & a_{22} & \cdots & a_{2,M^+}\\
  \vdots & \vdots & \vdots& \vdots\\
  a_{M^+,1} & a_{M^+,2} & \cdots & a_{M^+,M^+}
      \end{matrix}
\right) \left(\begin{matrix}
       c_{1,n}^{M,N,+}\\ c_{2,n}^{M,N,+}\\ \vdots \\ c_{M^+,n}^{M,N,+}
      \end{matrix}
\right) =\left(\begin{matrix}
       b_1\\b_2\\\vdots\\b_{M^+}
      \end{matrix}
\right),
\end{equation}
where
\[
 a_{j,\ell}=\left<\phi_\ell^{M,N,+},\phi_j^{M,N,+}\right>_{L^2(\Gamma_2^+)},\quad b_j=\left<\phi_n^{(h)},\phi_j^{M,N,+}\right>_{L^2(\Gamma_2^+)}.
\]
Note that since $\phi_m^{M,N,\pm}$ are generalized eigenfunctions, they are linearly independent. Thus the matrix above is invertible. However, this problem is always ill-posed when the number $m$ is large. Fortunately, from Theorem \ref{th:int_r} and Theorem \ref{th:est_eigenvalue}, we only need a small number of eigenfunctions due to the fast convergence rate thus the matrices are also small. The process to compute $c_{m,n}^{M,N,-}$ is similar. Then we obtain the decomposition
\[
 \left.\phi_n^{(h)}\right|_{\Gamma_2^+}=\sum_{m=1}^{M+}c^{M,N,+}_{m,n} \left.\phi_{m}^{M,N,+}(x)\right|_{\Gamma_2^+}\text{ and }\left.\phi_n^{(h)}\right|_{\Gamma_{-1}^-}=\sum_{m=1}^{M-}c^{M,N,-}_{m,n} \left.\phi_{m}^{M,N,-}(x)\right|_{\Gamma_{-1}^-}.
\]
Thus
\[
\left. \frac{\partial u_{h,M,N}}{\partial x_1}\right|_{\Gamma_2^+}=\sum_{j=1}^{N_h}\widehat{u}(j)\sum_{m=1}^{M^+}c^{N,M,+}_{m,j} \left.\frac{\partial \phi_{m}^{M,N,+}}{\partial x_1}(x)\right|_{\Gamma_2^+}:=\sum_{j=1}^{N_h}\widehat{u}(j) \psi_{j}^{M,N,+}(x)
\]
and
\[
\left. \frac{\partial u_{h,M,N}}{\partial x_1}\right|_{\Gamma_{-1}^-}=\sum_{j=1}^{N_h}\widehat{u}(j)\sum_{m=1}^{M^-}c^{N,M,-}_{m,j} \left.\frac{\partial \phi_{m}^{M,N,-}}{\partial x_1}(x)\right|_{\Gamma_{-1}^-}:=\sum_{j=1}^{N_h}\widehat{u}(j) \psi_{j}^{M,N,-}(x)
\]
where $\psi_{j}^{M,N,+}(x)=\sum_{m=1}^{M^+}c^{M,N,+}_{m,j} \left.\frac{\partial \phi_{m}^{M,N,+}}{\partial x_1}(x)\right|_{\Gamma_2^+}$ and $\psi_{j}^{M,N,-}(x)=\sum_{m=1}^{M^-}c^{M,N,-}_{m,j} \left.\frac{\partial \phi_{m}^{M,N,-}}{\partial x_1}(x)\right|_{\Gamma_{-1}^-}$ are computed before the finite element discretization.

Then the discretization of \eqref{eq:var_D_comput} is given as follows:
\begin{equation}
 \label{eq:var_Dh}
 \begin{aligned}
 \sum_{j=1}^{N_h}\widehat{u}(j)\int_D\left[\nabla \phi_\ell^{(h)}\cdot\nabla\overline{\phi_\ell^{(h)}}-k^2  n \phi_j^{(h)}\overline{\phi_\ell^{(h)}}\right]\d x
 &-\sum_{j=1}^{N_h}\widehat{u}(j)\int_{\Gamma_2^+} \psi_{j}^{M,N,+}\overline{\phi_\ell^{(h)}}\d s\\&+\sum_{j=1}^{N_h}\widehat{u}(j)\int_{\Gamma_2^-} \psi_{j}^{M,N,-}\overline{\phi_\ell^{(h)}}\d s=-\int_D f\overline{\phi_\ell^{(h)}}\d x.
 \end{aligned}
\end{equation}
By solving \eqref{eq:var_Dh}, we get the final result $u_{h,M,N}$.

The algorithm is organized as follows:
\begin{algorithm}[H]\label{alg}
\caption{Numerical simulation of the LAP solution of \eqref{eq:waveguide}.}
\begin{enumerate}
 \item Find out all the eigenfunctions $\phi_{m}^{M,N,\pm}$ by solving \eqref{eq:eig_pb} and the orthonormal process \eqref{eq:ortho}.
 \item Decide the coefficients $c_{j,\ell}^{M,N,\pm}$ by solving \eqref{eq:linear_sys}.
 \item Formulate \eqref{eq:var_Dh} from the computed eigenfunctions and coefficients.
\end{enumerate}
\end{algorithm}

\subsection{Error estimation}

The numerical analysis of Algorithm \ref{alg} consists of two parts. In the first part, we estimate the error of the eigenfunctions from \eqref{eq:eig_pb}; in the second part, we study the convergence of the finite element discretization \eqref{eq:var_Dh}.
First we make the further assumption for the refractive index.

\begin{assumption}
 \label{asp3}
 The refractive index $q$ is periodic in $\Omega_\pm$, strictly positive, and it is real analytic. 
\end{assumption}

We begin with the regularity of the solution of \eqref{eq:per} when $q$ satisfies the above assumption. From the periodicity boundary condition, the solution can be extended to the solution in the whole waveguide $W$. We summarize the Cauchy-Kowalesky theorem as follows.

\begin{theorem}[Proposition 4.2, \cite{Taylor2011}]
\label{th:CK}
 When $q$ and $g$ are real analytic functions in $W$, then the solution $v$ is real analytic. If $g=0$ in $W$, $v$ is analytic and $1$-periodic in  $x_1$  directions. 
\end{theorem}





Since $v$ is analytic and periodic, it is expanded into the  Fourier series \eqref{eq:v_four}. Since it is well known that the Fourier coefficients for periodic analytic functions (see Sect I.4, \cite{Katzn2004}), we have the following estimation:
\[
 \left|v_{j,\ell}\right|=O\left(\exp\left(-c\sqrt{j^2+\ell^2}\right)\right).
\]
Define the finite dimensional subspace
\[
 X_N:=\left\{\phi=\sum_{j=-N}^N\sum_{\ell=1}^N \phi_{j,\ell}^N \exp(\i 2\pi j x_1)\sin(\pi\ell x_2),\right\}\subset L^2(W_0)
\]
and let $P_N$ be the projection operator from $L^2(W_0)$ to $X_N$, then $v_N=P_N v$. Thus
\[
 \|v-v_N\|_{L^2(W_0)}= \|(I-P_N)v\|_{L^2(W_0)}=O(\exp(-cN)).
\]

We are prepared to study the error estimation of the solution of the linear eigenvalue problem \eqref{eq:eig_pb}. The numerical analysis for the approximation of eigenvalues and eigenfunctions has been studied in many papers. For details we refer to equations  (94)-(96) and Section 4 in \cite{Engst2014}.

Let $\mu$ be a discrete spectrum of the operator $\L$ with the ascent $n\in\N$, which means ${\rm ker}(\L-\mu I)^{n+1}={\rm ker}(\L-\mu I)^n\supset {\rm ker}(\L-\mu I)^{n-1}$. Let $m$ be the dimension of the eigenspace corresponding to the eigenvalue $\mu$. Then we get $n$ sequence of eigenvalues $\{\mu_j^N: \,N\in\N\}$ ($j=1,2,\dots,N$) such that
\[
 \max_{1\leq j\leq   n}\left|\mu-\mu_j^N\right|=O\left( \|(I-P_N)v\|_{L^2(W_0)}^2\right)=O\left(\exp\left(-2c N/n\right)\right).
\]
Let $v$ be a generalized eigenvector of $\L-\mu I$, then for any integer $\ell\in[j,n]$, there is a generalized eigenvector $v_N$ of order $\ell$ such that
\[
 \|v-v_N\|_V=O\left(\exp(-cN(\ell-j+1)/n)\right).
\]
Note that from Theorem \ref{th:est_eigenvalue}, when $|\Im(\mu)|$ is sufficiently large, the eigenspace related to $\mu$ is of dimension one. Thus we can always choose a suitable
$\gamma>0$ such that
\[
 \left|\mu-\mu_j^N\right|\leq C\exp(-\gamma N)\text{ and }\|v-v_N\|_V\leq C \exp(-\gamma N).
\]

Let $\K_r^+$ and $\K_R^-$ be defined by replacing the eigenfunctions by the numerical computation, then we finally get the following approximation:
\[
 \left\|\K_r^+-\K_r^N\right\|,\,\left\|\K_R^--\K_R^N\right\|
 \leq C\exp(-\gamma N).
\]
Let $r$ and $R$ be defined by a fixed positive integer $M$, then from Theorem \ref{th:int_r}, we finally have:
\[
  \left\|\K_0^+-\K_r^N\right\|,\,\left\|\K_\infty^--\K_R^N\right\|=O\left(\exp(-\gamma N)+\exp(-\xi M)\right).
\]

The following regularity result comes directly from the interior regularity for elliptic equations, for details we refer to Theorem 5, Section 6.3 in \cite{Evans1998}.

\begin{theorem}
 \label{th:regularity}
 Suppose $\partial \Omega$ is composed of finite number of non-intersecting $C^4$-curves. The refractive index $q$ satisfies Assumption \ref{asp1}, \ref{asp2} and \ref{asp3}, $f\in H^2(\Omega)$ and is compactly supported. Let $u$ be the solution of \eqref{eq:var_D} in $H^1(D)$, then it lies in the space $\widetilde{H}^4(D)$.
\end{theorem}

Based on above regularity results, we study the convergence of \eqref{eq:var_Dh} based on the finite element method. Define the finite dimensional subspace
\[
 Y_h:=\left\{\phi(x)=\sum_{\ell=1}^L \widehat{\phi}(\ell)\psi_\ell(x)\right\}\subset \widetilde{H}^1(D).
\]
Let $u_h$ be the solution of \eqref{eq:var_D} in $Y_h$, then from Theorem 4.7.3 in  \cite{Brenn1994}, we have the following error estimation:
\begin{equation}
 \label{eq:err_fem}
 \|u-u_h\|_{L^2(D)}\leq Ch^4\|u\|_{H^4(D)}
\end{equation}
Let $u_{N,M,h}$ be the solution of \eqref{eq:var_D_comput} where the DtN maps are chosen as above. Then
\begin{equation}
 \label{eq:err_total}
 \left\|u-u_{N,M,h}\right\|_{L^2(D)}\leq Ch^4+C e^{-\gamma N}+Ce^{-\xi M}.
\end{equation}

\section{Numerical examples}

To illustrate the efficiency of Algorithm \ref{alg}, we show two numerical examples. Note that to guarantee the sufficient regularity of the solutions, we only consider very smooth domains and  refractive indexes.

\begin{example}
\label{eg1}
 We consider the following problem in the planar waveguide $W=\R\times(0,1)$:
 \[
 \Delta u+k^2 q u=f\text{ in }W;\quad u=0\text{ on }\partial W.
\]
 The refractive index is represented by $q=q_1+q_2$, where $q_1$ is a $1$-periodic function:\begin{align*}
   q_1(x)=&2+\left[(2-4\i)e^{-14\pi\i x_1}+(2+4\i)e^{14\pi\i x_1}\right]\cos(\pi x_2)+3 \left[e^{-6\pi\i x_1}+ e^{6\pi\i x_1}\right]\cos(2\pi x_2)\\&+\left[(1+0.2\i)e^{-2\pi\i x_1}+(1-0.2\i)e^{2\pi\i x_1}\right]\cos(3\pi x_2).
 \end{align*}
 Both $q_2$ and $f$ are compactly supported functions:
\begin{equation*}
q_2(x)=\begin{cases}
		0,\quad |x-b_0|>0.15;\\
		2,\quad 0.1<|x-b_0|<0.15;\\
		2\zeta(|x-b_0|;0.1,0.15),\quad\text{otherwise;}
	\end{cases}\quad 	f(x)=\begin{cases}
	0,\quad |x-c_0|>0.3;\\
	3,\quad 0.1<|x-c_0|<0.3;\\
	3\,\zeta(|x-c_0|;0.1,0.3),\quad\text{otherwise;}
\end{cases}
\end{equation*}
where $b_0=(0.2,0.2)^\top$,  $c_0=(0.1,0.4)^\top$ and $\zeta(t)$ is a {$C^4$}-continuous function defined {by}
\begin{equation*}
	\zeta(t;a,b)=\begin{cases}
		1,\quad t\leq a;\\
		0,\quad t\geq b;\\
		1-\left[\int_{\tau=a}^b (\tau-a)^4(\tau-b)^4\d \tau\right]^{-1}\left[\int_{\tau=a}^t (\tau-a)^4(\tau-b)^4\d \tau\right],\, a<t<b.
	\end{cases}
\end{equation*}
For the visualization of the waveguide we refer to Figure \ref{fig:sample_eg1}
\end{example}

\begin{figure}[H]
\begin{center}
 \includegraphics[width=0.9\textwidth]{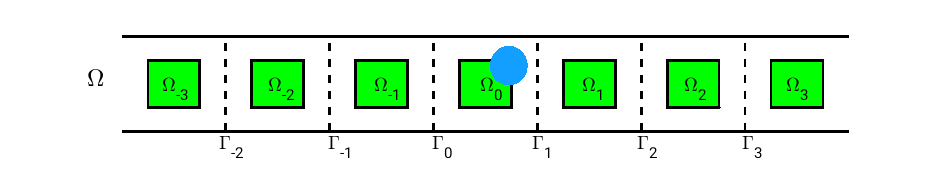}
 \caption{Structure of the waveguide for Example \ref{eg1}.}
 \label{fig:sample_eg1}
 \end{center}
\end{figure}

\begin{example}
 \label{eg2}
 The domain $\Omega$ is defined by three parts (see Figure \ref{fig:sample}). The left guide $\Omega_-=(-\infty,-2)\times(-0.5,0.5)$ and the refractive index is given by
 \[
   q_-(x_1,x_2)=q_1(x_1,x_2+0.5).
 \]
The right guide $\Omega_+=(0,+\infty)\times(-0.25,0.25)$ and the refractive index is
\begin{align*}
 q_+(x)=&3+\left[2\i e^{-20\pi\i x_1}-2\i e^{20\pi\i x_1}\right]\cos(4\pi(x_2+0.25))\\&+\left[(0.3-0.8\i)e^{-4\pi\i x_1}+(0.3+0.8\i)e^{4\pi\i x_1}\right]\cos(10\pi(x_2+0.25)).
\end{align*}
The domain $\Omega_0$ is an annulus with center $(-1,0)$, large radius $1$ and small radius. The domain $\Omega_0$ is connected with $\Omega_+$ and $\Omega_-$ by an open domain and then composes the whole domain $\Omega$. We also assume that the boundary of $\Omega$ is composed of three disjoint $C^4$-smooth boundaries. In particular, let the circle with center $(-1,0)$ and radius $0.5$ be denoted by $C_0$. Then the problem we are considering is 
\[
 \Delta u+k^2 q u=0\text{ in }\Omega;\quad u=0\text{ on }\partial\Omega\setminus C_0;\quad u=f\text{ on }C_0,
\]
where
\[
 f(x)=\exp\left(\i\cos(0.5)x_1+\i\sin(0.5)x_2\right).
\]
\end{example}

To compute reference solutions, we apply the recursive doubling procedure which was proposed in \cite{Ehrhardt2009a}. The method is based on the finite element method introduced in Section 5.2 with meshsize $0.005$. We first compute the DtN map from the recursive doubling procedure with an extrapolation technique, and then solve \eqref{eq:var_D} by the finite element method. The solution is denoted by $u_r$ and is treated as the ``exact solution''.

Then we compute the numerical results by Algorithm \ref{alg}. To guarantee the accuracy of the eigenvalues and eigenvectors, we fix $N=128$ and only vary the parameter $M$. First, we use a spectral method to find out all the eigenfunctions that are corresponding to eigenvalues that in $D_M^+$ or $D_M^-$. For the waveguide in Example \ref{eg1}, since the height of the waveguide is $1$, we choose $D_M^+:=[-\pi,\pi]+\i[0,M\pi)$ and $D_M^-=[-\pi,\pi]+\i(-M\pi,0]$; while for the left guide in Example \ref{eg2}, we still choose $D_M^-=[-\pi,\pi]+\i(-M\pi,0]$ but for the right guide, since the height is $0.5$, we choose $D_M^+:=[-\pi,\pi]+\i[0,2M\pi)$. For the finite element approximation, the meshsize $h$ is fixed to be $0.005$. We compare the relative norm between the numerical solution $u_M$ and the reference solution $u_r$:
\[
 \rm{error}:=\frac{\|u_M-u_r\|_{L^2(D)}}{\|u_r\|_{L^2(D)}}.
\]
For the results we refer to Table \ref{table:eg_error}. Note that we are only interested in the dependence of the error on the parameter $M$, since $N$ is already sufficiently large so the error brought by $N$ is ignored, while the dependence on $h$ is a standard topic in finite element method and it is not an important topic in this paper. The errors for both examples stay at a relatively low level. For the first example, the relative error stays around $4.5\times 10^{-4}$; for the second one, the relative error first decays significantly as $M$ increases, but then stays around $1.8\times 10^{-4}$. This fact can be explained by the error brought by the finite element method, or the recursive doubling procedure, or in other words, not explicit reference solutions.
For the visualization of the numerical result for Example 2, we refer to Figure \ref{fig:sample_eg2}.

\begin{table}[H]
\centering
 \label{table:eg_error}
 \begin{tabular}{|c|c|c|c|c|c|}
 \hline
  $M$ &$1$&$2$&$3$&$4$&$5$\\
  \hline
 Example 1  & $4.5$E$-4$ & $4.5$E$-4$ & $4.5$E$-4$ & $4.5$E$-4$&$4.5$E$-4$\\
 \hline
 Example 2 & $2.1$E$-3$&$8.2$E$-4$&$1.8$E$-4$&$1.9$E$-4$&$1.7$E$-4$\\
   \hline
 \end{tabular}
\caption{Relative errors of Examples.}
\end{table}

At the same time, we are also interested in the convergence rate of the algorithm with respect to the parameter $M$. To this end, we compute the numerical solution for $M=10$, and let $u_{10}$ to be the ``exact solution''. Then we compare the relative norm between the numerical solution $u_M$ and the reference solution $u_{10}$:
\[
 \rm{error}:=\frac{\|u_M-u_{10}\|_{L^2(D)}}{\|u_{10}\|_{L^2(D)}}.
\]
The relative errors are shown in  Table \ref{table:eg_conv} and the dependence of the logarithm of the errors and $M$ is shown in Figure \ref{fig:conv}. We can roughly see the linear dependence of the logarithm of the errors on the parameter $M$. Thus the error decays exponentially with respect to $M$. Finally we also want to mention that, the number of eigenfunctions depends linearly on $M$, i.e., for both half-guides, the number of eigenfunctions related to the eigenvalues lying in $D_M^+$ (or $D_M^-$) is always $M$. So we don't need to solve a large ill-posed linear system \eqref{eq:linear_sys} to achieve an accurate solution.

\begin{figure}[H]
\begin{center}
 \includegraphics[width=1.0\textwidth]{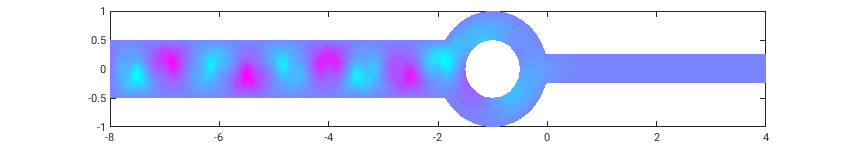}
 \caption{Numerical solution of Example \ref{eg2}.}
 \label{fig:sample_eg2}
 \end{center}
\end{figure}

\begin{table}[H]
\centering
 \label{table:eg_conv}
 \begin{tabular}{|c|c|c|c|c|}
 \hline
  $M$ &$2$&$4$&$6$&$8$\\
  \hline
 Example 1  & $6.7$E$-6$&$2.9$E$-6$&$2.4$E$-6$&$5.6$E$-7$\\
 \hline
 Example 2  & $8.2$E$-4$&$1.4$E$-4$&$8.2$E$-5$&$4.6$E$-5$\\
   \hline
 \end{tabular}
\caption{Convergence of the problem.}
\end{table}

\begin{figure}[H]
\begin{center}
\begin{tabular}{c c}
 \includegraphics[width=0.45\textwidth]{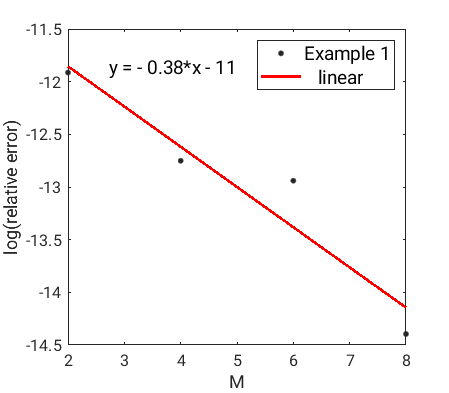}&\includegraphics[width=0.45\textwidth]{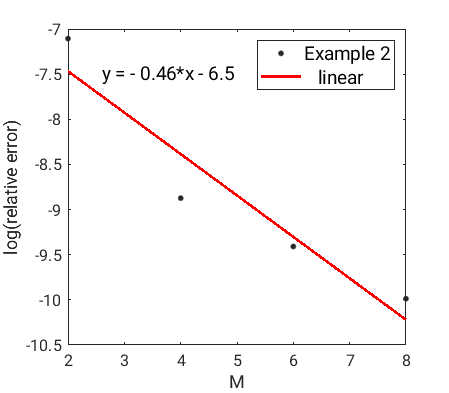}\\
 Example 1 & Example 2
\end{tabular}
 \caption{Convergence rates for Example 1 and 2.}
 \label{fig:conv}
 \end{center}
\end{figure}

\section*{Acknowledgment}
This work is funded by the Deutsche Forschungsgemeinschaft (DFG, German Research Foundation) – Project-ID 258734477 – SFB 1173

\bibliographystyle{plain}
\bibliography{ip-biblio}
\end{document}